\documentclass[12pt,reqno]{amsart}
\usepackage[a4paper,margin=2.5cm,top=2.5cm,bottom=2.5cm,centering,headheight=3ex,headsep=4ex,vcentering]{geometry}
\usepackage{amssymb,amsfonts,amsmath,amsthm}
\usepackage{cite}
\usepackage{orcidlink}
\usepackage[cal=cm,scr=euler]{mathalfa}
\usepackage{microtype}
\usepackage{mlmodern}
\numberwithin{equation}{section}
\theoremstyle{plain}
\newtheorem{theorem}{Theorem}
\newtheorem{conjecture}{Conjecture}
\newtheorem{lemma}{Lemma}

\theoremstyle{definition}

\theoremstyle{remark}

\usepackage{xcolor}
\usepackage{hyperref}
\definecolor{ao(english)}{rgb}{0.0, 0.5, 0.0}
\definecolor{pink}{rgb}{1.0, 0.0, 1.0}
\hypersetup{colorlinks=true, linkcolor=ao(english),citecolor=ao(english)}

\title[$q$-Analogues of some supercongruences ]{$q$-Analogues of some supercongruences related to Van Hamme's (E.2) and (F.2) supercongruences }

\author[Liton Karmakar]{Liton Karmakar\,\orcidlink{0009-0002-9532-1677}}
\address[Liton Karmakar]{Department of Mathematics, National Institute of Technology Silchar, Assam 788010, India}
\email{litonofficial8638@gmail.com}

\author[Arijit Jana]{Arijit Jana\,\orcidlink{0000-0002-7440-3297}}
\address[Arijit Jana]{Department of Mathematics, National Institute of Technology Silchar, Assam 788010, India}
\email{jana94arijit@gmail.com}


\begin{document}

\begin{abstract}
In this paper, we present a one-parameter extension of some supercongruences related to Van Hamme's
$\mathrm{(E.2)}$ and $\mathrm{(F.2)}$ supercongruences. To deduce these results, we first establish the
corresponding $q$-supercongruences modulo the third power of a cyclotomic
polynomial. Our proofs rely on the \emph{creative microscoping} method
devised by Guo and Zudilin~\cite{guo2019q}, together with Jackson's
${}_{6}\phi_{5}$ summation formula.

\end{abstract}

\maketitle

\section{Introduction and statement of results}

Throughout this paper, we adopt the standard notation of $q$-series
(see, e.g., Gasper and Rahman~\cite{gasper}). For complex numbers $a$ and $q$
with $|q|<1$, the \emph{$q$-shifted factorial} is defined by
\begin{align*}
(a;q)_{\infty} &:= \prod_{k=0}^{\infty}\bigl(1-aq^{k}\bigr),\\[2pt]
(a;q)_{N} &:= \frac{(a;q)_{\infty}}{(aq^{N};q)_{\infty}},
\qquad N\geq 1,
\end{align*}
with the convention $(a;q)_0:=1$. For brevity, we also employ the
compact notation:
\[
(a_1,a_2,\ldots,a_L;q)_n := (a_1;q)_n\,(a_2;q)_n\cdots(a_L;q)_n .
\]

Let $\Phi_m(q)$ denote the $m$-th \emph{cyclotomic polynomial}, given by
\[
\Phi_m(q) := \prod_{\substack{1\leq d\leq m\\ \gcd(d,m)=1}}
\left(q-e^{2\pi i d/m}\right),
\]
where $i^2=-1$. Moreover, the \emph{$q$-integer} $[m]=[m]_q$ is defined by
\[
[m] := 1+q+q^2+\cdots+q^{m-1} = \frac{1-q^{m}}{1-q}.
\]

Finally, we recall the notion of congruence for rational functions.
Two rational functions $A(q)$ and $B(q)$ in $q$ are said to be
\emph{congruent modulo} a polynomial $M(q)$, written as
\[
A(q)\equiv B(q) \pmod{M(q)},
\]
if the numerator of the reduced form of $A(q)-B(q)$ is divisible by
$M(q)$ in the polynomial ring $\mathbb{Z}[q]$.

\par In the theory of supercongruences, basic hypergeometric series play a vital role in establishing $q$-supercongruences. Following Gasper and Rahman \cite{gasper}, the basic hypergeometric ${}_{k+1}\phi_k$ series with $k+1$ upper parameters $a_1, \ldots, a_{k+1}$, $k$ lower parameters $b_1, \ldots, b_k$, base $q$, and argument $z$ are defined by
\begin{align*}
{}_{k+1}\phi_{k}
\left[
\begin{matrix}
a_1, a_2, \ldots, a_{k+1} \\
b_1, \ldots, b_k
\end{matrix}
; q, z
\right]
:= \sum_{n=0}^{\infty}
\frac{(a_1, a_2, \ldots, a_{k+1}; q)_n}
{(q, b_1, \ldots, b_k; q)_n}
z^n.
\end{align*}
In this paper, we will concentrate exclusively on Jackson\rq s ${}_{6}\phi_5$ summation (see \cite[Appendix (II.21)]{gasper}):
\begin{align} \label{Jackson}
{}_6\phi_5\left[
\begin{matrix}
a, & qa^{\frac12}, & -qa^{\frac12}, & b, & c, & q^{-N} \\
& a^{\frac12}, & -a^{\frac12}, & aq/b, & aq/c, & aq^{N+1}
\end{matrix}
~~; q, \frac{aq^{N+1}}{bc}
\right]
=
\frac{(aq, aq/bc;q)_N}
{(aq/b, aq/c;q)_N}.
\end{align}
\par Since Ramanujan's celebrated $1913$ letter to Hardy \cite{ramanujan}, infinite
series for $1/\pi$ have attracted considerable attention. Among the elegant
series of this kind are
\begin{align}
    & \sum_{n=0}^{\infty} (-1)^n (6n+1) \frac{(\tfrac{1}{3})_n^3}{n!^3}
      = \frac{3\sqrt{3}}{2\pi}, \label{E1} \\
    & \sum_{n=0}^{\infty} (-1)^n (8n+1) \frac{(\tfrac{1}{4})_n^3}{n!^3}
      = \frac{2\sqrt{2}}{\pi}, \label{F1}
\end{align}
where $(a)_0 = 1$ and $(a)_n = a(a+1)\cdots(a+n-1)$ for $n \geq 1$ denotes the
\emph{shifted factorial}. In $1997$, Van Hamme \cite{van2020some} conjectured a list of supercongruences relating to Ramanujan-type formulas for $1/\pi$ with the $p$-adic Gamma function (labeled $\mathrm{(A.2)}-$$\mathrm{(M.2)}$). In particular, the series \eqref{E1} and \eqref{F1} correspond to the following elegant truncated
congruences:
\begin{align} 
   & (\mathrm{E}.2)\qquad   \sum_{n=0}^{\frac{p-1}{3}} (-1)^n(6n+1) \frac{(\frac{1}{3})_n^3}{n!^3} \equiv p \pmod{p^3}, \quad \text{if } p \equiv 1 \ (\bmod\ 3), \label{E2} \\
   & (\mathrm{F}.2)\qquad   \sum_{n=0}^{\frac{p-1}{4}} (-1)^n(8n+1) \frac{(\frac{1}{4})_n^3}{n!^3} \equiv p (-1)^{\frac{p-1}{4}} \pmod{p^3}, \quad \text{if } p \equiv 1 \ (\bmod\ 4), \label{F2}
\end{align}
where $p$ is an odd prime. 
In $2015$, Swisher \cite{swisher} proved Van Hamme's supercongruences
\eqref{E2} and \eqref{F2} by means of Long's \cite{long1} method, and also established the
following companion congruences for the remaining residue classes:
\begin{align}
    \sum_{n=0}^{(2p-1)/3} (-1)^n (6n+1) \frac{(\tfrac{1}{3})_n^3}{n!^3}
      &\equiv (-1)^{(2p-1)/3}\, 2p \pmod{p^3}, && \text{for primes } p \equiv 2 \pmod{3},
      \label{E2 2} \\
    \sum_{n=0}^{(3p-1)/4} (-1)^n (8n+1) \frac{(\tfrac{1}{4})_n^3}{n!^3}
      &\equiv (-1)^{(3p-1)/4}\, 3p \pmod{p^3}, && \text{for primes } p \equiv 3 \pmod{4}.
      \label{F2 2}
\end{align}
The second author and Kalita~\cite[Theorem $1.1$]{jana1} further generalized \eqref{E2 2} and \eqref{F2 2}. On the other hand, motivated by the following analogue of (B.2),
\begin{align*} 
    \sum_{n=0}^{(p+1)/2} (-1)^n (4n-1) \frac{(-\tfrac{1}{2})_n^3}{n!^3}
      \equiv (-1)^{(p+1)/2}\, p \pmod{p^3},
\end{align*}
which holds for every odd prime $p$ and is a special case of
\cite[Theorem $1.3$]{guo19} (see also \cite[Theorem $4.9$]{guo2019q}),
the second author \cite{jana2} proved, by means of hypergeometric series identities
and evaluations, that
\begin{align*}
    \sum_{n=0}^{(p+1)/3} (-1)^n (6n-1) \frac{(-\tfrac{1}{3})_n^3}{n!^3}
      &\equiv (-1)^{(p+1)/3}\, p \pmod{p^3}
      && \text{for primes } p \equiv 2 \pmod{3}, \\
    \sum_{n=0}^{(p+1)/4} (-1)^n (8n-1) \frac{(-\tfrac{1}{4})_n^3}{n!^3}
      &\equiv (-1)^{(p+1)/4}\, p \pmod{p^3}
      && \text{for primes } p \equiv 3 \pmod{4}.
\end{align*}

In $2022$, Guo \cite{guo2022new} introduced a new extension of the $\mathrm{(A.2)}$ supercongruence of Van Hamme, which includes an extra parameter $s$. In particular, Guo explored an intriguing $q$-supercongruence
with an additional parameter $s$ : for any odd integer $m>1$ and
$0\le s\le (m-1)/2$, modulo $\Phi_m(q)^3$,
\begin{align*}
&\sum_{n=s}^{(m-1)/2}
(-1)^n[4n+1]
\frac{
(q;q^2)_{n-s}(q;q^2)_{n+s}
(q;q^2)_n^2(q^2;q^4)_n
}{
(q^2;q^2)_{n-s}(q^2;q^2)_{n+s}
(q^2;q^2)_n^2(q^4;q^4)_n
}
q^n
\\
&\equiv
\begin{cases}
\displaystyle
[m] q^{s(m-s)}
\frac{
(q;q^2)_s^2
(q^2;q^4)_{(m-2s-1)/4}^2
}{
(q^4;q^4)_{(m+2s-1)/4}^2
},
&
\displaystyle
\text{if }s\equiv \frac{m-1}{2}\pmod{2},
\\
0,
&
\displaystyle
\text{otherwise },
\end{cases}
\end{align*}
which simplifies to Van Hamme\rq s $\mathrm{(A.2)}$ supercongruence
when $q\to 1$ and $s=0$. Motivated by his work, many researchers have given one-parameter extensions of some Van Hamme\rq s remarkable supercongruences; see \cite{Guo-qB2C2, Guo25, HeWang2023_D2, Liton2025 Ramanujan, KJ 2026 RJ, P2 Liton, LiWang26, NiWang E2 2025, Tang G2, Tang2023 JMAA, Wang2024 Results Math, Wei2024 RACSAM}. 
\par Following these, we here establish a one-parameter generalization of the supercongruences related to  \eqref{E2} and \eqref{E2 2} involving a hypergeometric sign reversal parameter. More precisely, for any prime $p$ satisfying $p\equiv 3-4t \pmod{6}$ with $t \in \{1,2\}$, and $0\leq s\leq(tp+1)/3$, we establish the following result:
\begin{align} \label{P8 E2 result}
    \sum_{n=s}^{\frac{tp+1}{3}} (-1)^n (6n-1) \frac{(-\frac{1}{3})_n^2 (-\frac{1}{3})_{n+s}}{n!^2 (n-s)!} \equiv (-1)^{\frac{tp+1}{3}} (tp+3s) \frac{(-\frac{1}{3})_s^2 (-\frac{1}{3})_{\frac{tp+1}{3}+s}}{s!^2 (-\frac{1}{3})_{\frac{tp+1}{3}-s} } \pmod{p^3}.
\end{align}
Our approach to proving the congruence \eqref{P8 E2 result} is based on establishing the following $q$-supercongruence.
\begin{theorem} \label{P8 Thm 1}
    Let $m>1$ be an integer and $s \le (tm+1)/3$ a non-negative integer with $t \in \{1,2\}$ satisfying $m \equiv 3-4t \pmod{6}$. Then, modulo $\Phi_m(q)^3,$

\begin{align} \label{P8 Thm 1 eqn}
\sum_{n=s}^{\frac{tm+1}{3}} &(-1)^n [6n-1]
\frac{(q^{-1};q^3)_n^2 (q^{-1};q^3)_{n+s} (-q^{-1};q^3)_{n-s}}
{(q^3;q^3)_n^2 (-q^3;q^3)_{n+s} (q^3;q^3)_{n-s}}
q^{5n} \notag \\
\equiv
&(-1)^{\frac{tm+1}{3}} q^{4s+\frac{tm+1}{3}} [tm+3s] 
\frac{(q^{-1};q^3)_s^2 (q^{-1};q^3)_{\frac{tm+1}{3}+s} (-q^{-2};q^3)_{\frac{tm+1}{3} - s}}
{(q^3;q^3)_s^2 (-q^3;q^3)_{\frac{tm+1}{3}+s} (q^{-1};q^3)_{\frac{tm+1}{3} - s}}.
\end{align}
\end{theorem}

The next objective of this paper is to establish the following parametric generalization of the supercongruences related to \eqref{F2} and \eqref{F2 2} that involve a hypergeometric sign reversal parameter. More precisely, for any prime $p$ satisfying $p \equiv -t \pmod{4}$ with $t \in \{1,3\},$ and $0\leq s\leq(tp+1)/4$, we prove that
\begin{align} \label{P8 F2 result}
    \sum_{n=s}^{\frac{tp+1}{4}} (-1)^n (8n-1) \frac{(-\frac{1}{4})_n^2 (-\frac{1}{4})_{n+s}}{n!^2 (n-s)!} \equiv (-1)^{\frac{tp+1}{4}} (tp+4s) \frac{(-\frac{1}{4})_s^2 (-\frac{1}{4})_{\frac{tp+1}{4}+s}}{s!^2 (-\frac{1}{4})_{\frac{tp+1}{4}-s} } \pmod{p^3}.
\end{align}
The supercongruence \eqref{P8 F2 result} follows directly from the corresponding $q$-supercongruence.
\begin{theorem} \label{P8 Thm 2}
    Let $m>1$ be an integer and $s \le (tm+1)/4$ a non-negative integer with $t \in \{1,3\}$ satisfying $m \equiv -t \pmod{4}$. Then, modulo $\Phi_m(q)^3,$

\begin{align} \label{P7 Thm 2 eqn}
\sum_{n=s}^{\frac{tm+1}{4}} &(-1)^n [8n-1]
\frac{(q^{-1};q^4)_n^2 (q^{-1};q^4)_{n+s} (-q^{-1};q^4)_{n-s}}
{(q^4;q^4)_n^2 (-q^4;q^4)_{n+s} (q^4;q^4)_{n-s}}
q^{6n} \notag \\
\equiv
&(-1)^{\frac{tm+1}{4}} q^{5s+\frac{tm+1}{4}} [tm+4s] 
\frac{(q^{-1};q^4)_s^2 (q^{-1};q^4)_{\frac{tm+1}{4}+s} (-q^{-2};q^4)_{\frac{tm+1}{4} - s}}
{(q^4;q^4)_s^2 (-q^4;q^4)_{\frac{tm+1}{4}+s} (q^{-1};q^4)_{\frac{tm+1}{4} - s}}.
\end{align}
\end{theorem}

The rest of the paper is structured as follows. In Section \ref{Sec2}, we state two key lemmas that are crucial for proving our main results. In Sections \ref{Sec 3} and \ref{Sec 4}, we prove Theorems \ref{P8 Thm 1} and \ref{P8 Thm 2} by deriving their generalized parametric forms with an additional parameter $a$. Finally, we conclude the paper by proposing two open problems in Section \ref{Sec 5} for interested readers.

\section{Preliminary lemmas} \label{Sec2}
In this section, first, we state the following two lemmas, which will be used in the proofs of our main results. 
\begin{lemma} [Lemma $2.1$, \cite{GuoSchlosser Isr J}]  \label{P8 Lemma 1}
    Let $d$, $M$ and $m$ be positive integers with $M\le m-1$, and $h$ an integer satisfying
$dM\equiv -h \pmod m$. Then, for $0\le n\le M$, we have
\begin{align} \label{P8 lemma 1 Eqn}
    \frac{(aq^h;q^d)_{M-n}}{(q^d/a;q^d)_{M-n}}
\equiv
(-a)^{M-2n} q^{M(dM-d+2h)/2+(d-h)n}
\frac{(aq^h;q^d)_n}{(q^d/a;q^d)_n}
\pmod{\Phi_m(q)}.
\end{align}
\end{lemma}

\begin{lemma} \label{P8 Lemma 2}
    Let $d$, $M$ and $m$ be positive integers with $M\le m-1$, and $h$ an integer satisfying
$dM\equiv -h \pmod m$. Then, for a non-negative integer $s$ such that $s\le n\le M-s$, we have

\begin{align} \label{P8 lemma 2 Eqn 1}
    \frac{(aq^h;q^d)_{M-n-s}}{(q^d/a;q^d)_{M-n+s}}
\equiv
(-a)^{M-2n} q^{M(dM-d+2h)/2+(d-h)n + hs + 2dns}
\frac{(aq^h;q^d)_{n-s}}{(q^d/a;q^d)_{n+s}}
\pmod{\Phi_m(q)},
\end{align}
and
\begin{align} \label{P8 lemma 2 Eqn 2}
    \frac{(aq^h;q^d)_{M-n+s}}{(q^d/a;q^d)_{M-n-s}}
\equiv
(-a)^{M-2n} q^{M(dM-d+2h)/2+(d-h)n - hs - 2dns}
\frac{(aq^h;q^d)_{n+s}}{(q^d/a;q^d)_{n-s}}
\pmod{\Phi_m(q)}.
\end{align}
\end{lemma}
\begin{proof}
    Observing $q^{dM+h} \equiv q^m \equiv 1 \pmod{ \Phi_m(q)}$ and exploiting Lemma \ref{P8 Lemma 1}, it is routine to find that
    \begin{align*} 
&\displaystyle \frac{(aq^h;q^d)_{M-n-s}}
{(q^d/a;q^d)_{M-n+s}} \notag \\
&=
\frac{ \frac{(aq^h;q^d)_{M-n-s}}{(q^d/a;q^d)_{M-n-s}}          }
{(1- \frac{q^{d+dM-dn-ds}}{a})
 (1- \frac{q^{2d+dM-dn-ds}}{a})
 \cdots
 (1- \frac{q^{dM-dn+ds}}{a})} \notag \\
&\equiv
\frac{
(-a)^{M-2n-2s }
q^{M(dM-d+2h)/2+(d-h)(n+s)} \frac{(aq^h;q^d)_{n+s}}{(q^d/a;q^d)_{n+s}} }
{(1- \frac{q^{d-h-dn-ds}}{a})
 (1- \frac{q^{2d-h-dn-ds}}{a})
 \cdots
 (1- \frac{q^{-h-dn+ds}}{a})} \notag \\
&\equiv
(-a)^{M-2n}
q^{M(dM-d+2h)/2+(d-h)n + hs + 2dns}
\frac{(aq^h;q^d)_{n-s}}
{(q^d/a;q^d)_{n+s}}
\pmod{\Phi_m(q)} .
\end{align*} 
This proves \eqref{P8 lemma 2 Eqn 1}. Following this, one can prove \eqref{P8 lemma 2 Eqn 2}.
\end{proof}

\section{Proof of Theorem \ref{P8 Thm 1} } \label{Sec 3}
In order to prove Theorem \ref{P8 Thm 1}, we first establish the corresponding parametric version below.
\begin{theorem} \label{P8 lemma 3}
     Let $m>1$ be an integer and $s \le (tm+1)/3$ a non-negative integer with $t \in \{1,2\}$ satisfying $m \equiv 3-4t \pmod{6}$. Then, modulo $\Phi_m(q)(1-aq^{tm}) (a-q^{tm}),$
\begin{align} \label{P8 lemma 3 eqn}
\sum_{n=s}^{\frac{tm+1}{3}} &(-1)^n [6n-1]
\frac{(aq^{-1};q^3)_n (q^{-1}/a;q^3)_n (q^{-1};q^3)_{n+s} (-q^{-1};q^3)_{n-s}}
{(aq^3;q^3)_n (q^3/a;q^3)_n (-q^3;q^3)_{n+s} (q^3;q^3)_{n-s}}
q^{5n} \notag \\
\equiv
&(-1)^{\frac{tm+1}{3}} q^{4s+\frac{tm+1}{3}} [tm+3s] 
\frac{(aq^{-1};q^3)_s (q^{-1}/a;q^3)_s (q^{-1};q^3)_{\frac{tm+1}{3}+s} (-q^{-2};q^3)_{\frac{tm+1}{3} - s}}
{(aq^3;q^3)_s (q^3/a;q^3)_s (-q^3;q^3)_{\frac{tm+1}{3}+s} (q^{-1};q^3)_{\frac{tm+1}{3} - s}}.
\end{align}
\end{theorem}
\begin{proof}
    For $a=q^{-tm}$ or $a=q^{tm}$, the left-hand side of \eqref{P8 lemma 3 eqn} can be rewritten as
    \begin{align} \label{P8 lemma 3 eqn 2}
     &\sum_{n=s}^{\frac{tm+1}{3}} (-1)^n [6n-1]
\frac{(q^{-1-tm};q^3)_n (q^{-1+tm};q^3)_n (q^{-1};q^3)_{n+s} (-q^{-1};q^3)_{n-s}}
{(q^{3-tm};q^3)_n (q^{3+tm};q^3)_n (-q^3;q^3)_{n+s} (q^3;q^3)_{n-s}}
q^{5n} \notag \\
&= \sum_{n=0}^{\frac{tm+1}{3}-s} (-1)^{n+s} [6n + 6s-1]
\frac{(q^{-1-tm};q^3)_{n+s} (q^{-1+tm};q^3)_{n+s} (q^{-1};q^3)_{n+2s} (-q^{-1};q^3)_{n}}
{(q^{3-tm};q^3)_{n+s} (q^{3+tm};q^3)_{n+s} (-q^3;q^3)_{n+2s} (q^3;q^3)_{n}}
q^{5n+5s} \notag \\
& = (-1)^s [6s-1] \frac{(q^{-1-tm};q^3)_s (q^{-1+tm};q^3)_s (q^{-1};q^3)_{2s}}{(q^{3-tm};q^3)_{s} (q^{3+tm};q^3)_{s} (-q^3;q^3)_{2s}} q^{5s} \notag\\
&\qquad\times
{}_6\phi_5\left[
\begin{matrix}
q^{6s-1}, & q^{3s+\frac{5}{2}}, & -q^{3s+\frac{5}{2}},
& -q^{-1}, & q^{3s-1+tm}, & q^{3s-1-tm} \\
q^{3s-\frac{1}{2}}, & -q^{3s-\frac{1}{2}},
& -q^{6s+3}, & q^{3s+3-tm}, & q^{3s+3+tm}
\end{matrix}
; q^3, -q^5
\right].
    \end{align}
Making use of the replacements $q \mapsto q^3$, $a = q^{6s-1}$, $b= -q^{-1}$, $c = q^{3s-1+tm}$ and $N \mapsto \frac{tm+1}{3} - s$ in Jackson's ${}_6\phi_5$ summation formula \eqref{Jackson}, it is easy to evaluate the right-hand side of \eqref{P8 lemma 3 eqn 2} as
\begin{align*}
    & (-1)^s[6s-1]q^{5s}
\frac{(q^{-1-tm};q^3)_s (q^{-1+tm};q^3)_s (q^{-1};q^3)_{2s}}
{(q^{3-tm};q^3)_s (q^{3+tm};q^3)_s (-q^3;q^3)_{2s}}  \\
&\quad \times
\frac{(q^{6s+2};q^3)_{\frac{tm+1}{3}-s} (-q^{3s+4-tm};q^3)_{\frac{tm+1}{3}-s}}
{(-q^{6s+3};q^3)_{\frac{tm+1}{3}-s} (q^{3s+3-tm};q^3)_{\frac{tm+1}{3}-s}} \\
&= (-1)^{\frac{tm+1}{3}} [tm+3s] q^{4s + \frac{tm+1}{3}} \\
& \qquad \times \frac{(q^{-1-tm};q^3)_s (q^{-1+tm}; q^3)_s (q^{-1};q^3)_{\frac{tm+1}{3}+s} (-q^{-2};q^3)_{\frac{tm+1}{3}-s}}{(q^{3-tm};q^3)_s (q^{3+tm};q^3)_s (-q^3;q^3)_{\frac{tm+1}{3}+s} (q^{-1};q^3)_{\frac{tm+1}{3}-s}}.
\end{align*}
Thus, this proves that $q$-supercongruence \eqref{P8 lemma 3 eqn} is true modulo $1-aq^{tm}$ and $a-q^{tm}$. 

In light of lemmas \ref{P8 Lemma 1} and \ref{P8 Lemma 2}, it is routine to check that, for $d=3, h=-1, M=\frac{tm+1}{3}$, and $s \le n \le M-s$,
$$ E_q(M-n) \equiv - E_q(n) \pmod{\Phi_m(q)},$$
where $E_q(n)$ represents the following expression:
\begin{align*}
    E_{q}(n)= (-1)^n [6n-1]
\frac{(aq^{-1};q^3)_n (q^{-1}/a;q^3)_n (q^{-1};q^3)_{n+s} (-q^{-1};q^3)_{n-s}}
{(aq^3;q^3)_n (q^3/a;q^3)_n (-q^3;q^3)_{n+s} (q^3;q^3)_{n-s}}
q^{5n}.
\end{align*}
This means that the partial sum $\sum_{n=s}^{\frac{tm+1}{3}-s} E_q(n)$ is congruent to $0$ modulo $\Phi_m(q).$ For $  \frac{tm+1}{3}-s < n \le \frac{tm+1}{3},$ the factor $1-q^{tm}$ appear in the $q$-shifted factorial $(q^{-1};q^3)_{n+s}$, while the denominator of $E_q(n)$ is relatively prime to $\Phi_m(q)$. As a result, for each $n$ within this bound, $E_q(n)$ is congruent to $0$ modulo $\Phi_m(q)$. Since
\begin{align} \label{P8 R 1}
    [tm+3s] (q^{-1};q^3)_{\frac{tm+1}{3}+s} = [tm] (q^{-1};q^3)_{\frac{tm+1}{3}} (q^{tm+3};q^3)_s \equiv 0 \pmod{\Phi_m(q)}.
\end{align}
Combining \eqref{P8 R 1} with \eqref{P8 lemma 3 eqn}, one observes that the $q$-congruence \eqref{P8 lemma 3 eqn} is true modulo $\Phi_m(q)$. Since the polynomials $\Phi_m(q)$, $1-aq^{tm}$, and $a-q^{tm}$ are known to be pairwise coprime. Thus, by the Chinese remainder theorem for coprime polynomials, we conclude that \eqref{P8 lemma 3 eqn} holds modulo $\Phi_m(q)(1-aq^{tm})(a-q^{tm})$. This completes the proof of Theorem \ref{P8 lemma 3}.
\end{proof}
\begin{proof} [Proof of Theorem \ref{P8 Thm 1}]
 The polynomial $\Phi_m(q)(1-aq^{tm})(a-q^{tm})$ contains the factor $\Phi_m(q)^3$ when $a=1.$ Finally, taking $a=1$ in \eqref{P8 lemma 3 eqn}, we obtain the desired $q$-congruence \eqref{P8 Thm 1 eqn}.
\end{proof}
\section{Proof of Theorem \ref{P8 Thm 2}} \label{Sec 4}
Similarly, we must first demonstrate the generalized version of \eqref{P7 Thm 2 eqn} to prove Theorem \ref{P8 Thm 2}.
\begin{theorem} \label{P8 lemma 4}
     Let $m>1$ be an integer and $s \le (tm+1)/4$ a non-negative integer with $t \in \{1,3\}$ satisfying $m \equiv -t \pmod{4}$. Then, modulo $\Phi_m(q)(1-aq^{tm}) (a-q^{tm}),$
\begin{align} \label{P8 lemma 4 eqn}
\sum_{n=s}^{\frac{tm+1}{4}} &(-1)^n [8n-1]
\frac{(aq^{-1};q^4)_n (q^{-1}/a;q^4)_n (q^{-1};q^4)_{n+s} (-q^{-1};q^4)_{n-s}}
{(aq^4;q^4)_n (q^4/a;q^4)_n (-q^4;q^4)_{n+s} (q^4;q^4)_{n-s}}
q^{6n} \notag \\
\equiv
&(-1)^{\frac{tm+1}{4}} q^{5s+\frac{tm+1}{4}} [tm+4s] 
\frac{(aq^{-1};q^4)_s (q^{-1}/a;q^4)_s (q^{-1};q^4)_{\frac{tm+1}{4}+s} (-q^{-2};q^4)_{\frac{tm+1}{4} - s}}
{(aq^4;q^4)_s (q^4/a;q^4)_s (-q^4;q^4)_{\frac{tm+1}{4}+s} (q^{-1};q^4)_{\frac{tm+1}{4} - s}}.
\end{align}
\end{theorem}

\begin{proof} 
    For $a=q^{-tm}$ or $a=q^{tm}$, the left-hand side of \eqref{P8 lemma 4 eqn} can be rewritten as
    \begin{align} \label{P7 lemma 8 eqn 2}
     &\sum_{n=s}^{\frac{tm+1}{4}} (-1)^n [8n-1]
\frac{(q^{-1-tm};q^4)_n (q^{-1+tm};q^4)_n (q^{-1};q^4)_{n+s} (-q^{-1};q^4)_{n-s}}
{(q^{4-tm};q^4)_n (q^{4+tm};q^4)_n (-q^4;q^4)_{n+s} (q^4;q^4)_{n-s}}
q^{6n} \notag \\
&= \sum_{n=0}^{\frac{tm+1}{4}-s} (-1)^{n+s} [8n + 8s-1]
\frac{(q^{-1-tm};q^4)_{n+s} (q^{-1+tm};q^4)_{n+s} (q^{-1};q^4)_{n+2s} (-q^{-1};q^4)_{n}}
{(q^{4-tm};q^4)_{n+s} (q^{4+tm};q^4)_{n+s} (-q^4;q^4)_{n+2s} (q^4;q^4)_{n}}
q^{6n+6s} \notag \\
& = (-1)^s [8s-1] \frac{(q^{-1-tm};q^4)_s (q^{-1+tm};q^4)_s (q^{-1};q^4)_{2s}}{(q^{4-tm};q^4)_{s} (q^{4+tm};q^4)_{s} (-q^4;q^4)_{2s}} q^{6s} \notag\\
&\qquad\times
{}_6\phi_5\left[
\begin{matrix}
q^{8s-1}, & q^{4s+\frac{7}{2}}, & -q^{4s+\frac{7}{2}},
& -q^{-1}, & q^{4s+tm-1}, & q^{4s-tm-1} \\
q^{4s-\frac{1}{2}}, & -q^{4s-\frac{1}{2}},
& -q^{8s+4}, & q^{4s+4-tm}, & q^{4s+4+tm}
\end{matrix}
; q^4, -q^6
\right].
    \end{align}
Making use of the replacements $q \mapsto q^4$, $a = q^{8s-1}$, $b= -q^{-1}$, $c = q^{4s+tm-1}$ and $N \mapsto \frac{tm+1}{4} - s$ in Jackson's ${}_6\phi_5$ summation formula \eqref{Jackson}, it is easy to evaluate the right-hand side of \eqref{P7 lemma 8 eqn 2} as
\begin{align*}
    & (-1)^s[8s-1]q^{6s}
\frac{(q^{-1-tm};q^4)_s (q^{-1+tm};q^4)_s (q^{-1};q^4)_{2s}}
{(q^{4-tm};q^4)_s (q^{4+tm};q^4)_s (-q^4;q^4)_{2s}}  \\
&\quad \times
\frac{(q^{8s+3};q^4)_{\frac{tm+1}{4}-s} (-q^{4s+5-tm};q^4)_{\frac{tm+1}{4}-s}}
{(-q^{8s+4};q^4)_{\frac{tm+1}{4}-s} (q^{4s+4-tm};q^4)_{\frac{tm+1}{4}-s}} \\
&= (-1)^{\frac{tm+1}{4}} [tm+4s] q^{5s + \frac{tm+1}{4}} \\
& \qquad \times \frac{(q^{-1-tm};q^4)_s (q^{-1+tm}; q^4)_s (q^{-1};q^4)_{\frac{tm+1}{4}+s} (-q^{-2};q^4)_{\frac{tm+1}{4}-s}}{(q^{4-tm};q^4)_s (q^{4+tm};q^4)_s (-q^4;q^4)_{\frac{tm+1}{4}+s} (q^{-1};q^4)_{\frac{tm+1}{4}-s}}.
\end{align*}
Thus, this proves that $q$-supercongruence \eqref{P8 lemma 4 eqn} is true modulo $1-aq^{tm}$ and $a-q^{tm}$.

In view of Lemmas \ref{P8 Lemma 1} and \ref{P8 Lemma 2}, it is easy to see that, for $d=4$, $h=-1$, $M=\frac{tm+1}{4}$, and $s \le n \le M-s$,
\begin{align*}
    & (-1)^{M-n}[8(M-n)-1]
\frac{(aq^{-1};q^4)_{M-n}(q^{-1}/a;q^4)_{M-n}(q^{-1};q^4)_{M-n+s}(-q^{-1};q^4)_{M-n-s}}
{(aq^4;q^4)_{M-n}(q^4/a;q^4)_{M-n}(-q^4;q^4)_{M-n+s}(q^4;q^4)_{M-n-s}}
q^{6M-6n} \\
&\equiv - (-1)^n [8n-1]
\frac{(aq^{-1};q^4)_n (q^{-1}/a;q^4)_n (q^{-1};q^4)_{n+s} (-q^{-1};q^4)_{n-s}}
{(aq^4;q^4)_n (q^4/a;q^4)_n (-q^4;q^4)_{n+s} (q^4;q^4)_{n-s}}
q^{6n} \pmod{\Phi_m(q)}.
\end{align*}
This implies that
\begin{align*}
    & \sum_{n=s}^{\frac{tm+1}{4}-s} (-1)^n [8n-1]
\frac{(aq^{-1};q^4)_n (q^{-1}/a;q^4)_n (q^{-1};q^4)_{n+s} (-q^{-1};q^4)_{n-s}}
{(aq^4;q^4)_n (q^4/a;q^4)_n (-q^4;q^4)_{n+s} (q^4;q^4)_{n-s}}
q^{6n} \equiv 0 \pmod{\Phi_m(q)}.
\end{align*}
Thus, for each $n$ within the range $\left(\frac{tm+1}{4}-s, \frac{tm+1}{4} \right]$, the summand on the left-hand side of \eqref{P8 lemma 4 eqn} is congruent to $0$ modulo $\Phi_m(q)$, since the $q$-shifted factorial $(q^{-1};q^4)_{n+s}$ in the numerator contains the factor $(1-q^{tm})$, while its denominator is relatively prime to $\Phi_m(q)$. Furthermore, we observe that $[tm+4s] (q^{-1};q^4)_{\frac{tm+1}{4}+s} = [tm] (q^{-1};q^4)_{\frac{tm+1}{4}} (q^{tm+4};q^4)_s$ is congruent to $0$ modulo $\Phi_m(q)$. Therefore, the $q$-supercongruence \eqref{P8 lemma 4 eqn} is true modulo $\Phi_m(q)$. Since the polynomials $\Phi_m(q),$ $1-aq^{tm},$ and $a-q^{tm}$ are pairwise coprime. Thus, by the Chinese remainder theorem for coprime polynomials, the $q$-supercongruence \eqref{P8 lemma 4 eqn} holds modulo $\Phi_m(q)(1-aq^{tm}) (a-q^{tm})$. This completes the proof.
\end{proof}
\begin{proof} [Proof of Theorem \ref{P8 Thm 2}]
  Setting $a=1$ in Theorem \ref{P8 lemma 4}, the proof of Theorem \ref{P8 Thm 2} is immediate.
\end{proof}
\section{Concluding Remarks} \label{Sec 5}
We believe that the stronger version of Theorems \ref{P8 Thm 1} and \ref{P8 Thm 2}  should also be true. Based on the numerical evaluation, we propose the following conjectures as an open problem.
\begin{conjecture} \label{P8 Conj 1}
    The supercongruence \eqref{P8 Thm 1 eqn} is true modulo $[m]\Phi_m(q)^2$.
\end{conjecture}
Setting $m=p^r$ and the limiting case as $q \to 1$ in Conjecture \ref{P8 Conj 1}, we arrive at the following generalization of \eqref{P8 E2 result}: for any prime $p$ and non-negative integer $r$ satisfying
$p^r \equiv 3-4t \pmod 6$ with $t \in \{1,2\}$, and $0\le s\le \frac{tp^r+1}{3}$, we have
\begin{align*} 
    \sum_{n=s}^{\frac{tp^r+1}{3}} (-1)^n (6n-1) \frac{(-\frac{1}{3})_n^2 (-\frac{1}{3})_{n+s}}{n!^2 (n-s)!} \equiv (-1)^{\frac{tp^r+1}{3}} (tp^r+3s) \frac{(-\frac{1}{3})_s^2 (-\frac{1}{3})_{\frac{tp^r+1}{3}+s}}{s!^2 (-\frac{1}{3})_{\frac{tp^r+1}{3}-s} } \pmod{p^{r+2}}.
\end{align*}
\begin{conjecture} \label{P8 Conj2}
    The supercongruence \eqref{P7 Thm 2 eqn} is true modulo $[m]\Phi_m(q)^2$.
\end{conjecture}
Similarly, taking $m=p^r$ and then letting $q \to 1$ in Conjecture \ref{P8 Conj2}, we obtain a generalization of \eqref{P8 F2 result} as follows:
\begin{align*} 
    \sum_{n=s}^{\frac{tp^r+1}{4}} (-1)^n (8n-1) \frac{(-\frac{1}{4})_n^2 (-\frac{1}{4})_{n+s}}{n!^2 (n-s)!} \equiv (-1)^{\frac{tp^r+1}{4}} (tp^r+4s) \frac{(-\frac{1}{4})_s^2 (-\frac{1}{4})_{\frac{tp^r+1}{4}+s}}{s!^2 (-\frac{1}{4})_{\frac{tp^r+1}{4}-s} } \pmod{p^{r+2}},
\end{align*}
where $p$ is a prime and $r$ is a non-negative integer satisfying $p^r \equiv -t \pmod {4}$ with $t \in \{1,3\}$.

\section*{Declarations}
	
{\bf Author contributions:} All authors have contributed equally to the preparation of this manuscript.\\
{\bf Funding:} No funding was received for this research.\\
{\bf Conflicts of Interest:} The authors do not have any conflicts
of interest. \\
{\bf Data Availability Statement:}
	No datasets were generated or analyzed in this study. The results are obtained entirely through theoretical and mathematical methods.

\end{document}